\documentclass[10pt]{amsart}
\usepackage{amssymb, amscd, amsmath, amsthm, amscd,curves,
epsf, epsfig, latexsym, enumerate, pstricks,graphics}

\newtheorem*{thm}{Theorem}
\newtheorem*{lem}{Lemma}
\newtheorem*{cor*}{Corollary}

\begin{document}

\title[ Some questions on subgroups of $PD_3$-groups]{Some questions on subgroups of 3-dimensional Poincar\'e duality groups}
 
\author{J.A.Hillman}
\address{School of Mathematics and Statistics\\
     University of Sydney, NSW 2006\\
      Australia }

\email{jonathan.hillman@sydney.edu.au}

\begin{abstract}
Poincar\'e duality complexes model the homotopy types of closed manifolds.
In the lowest dimensions the correspondence is precise: 
every connected $PD_n$-complex  is homotopy equivalent 
to $S^1$ or to a closed surface, 
when $n=1$ or 2.
Every $PD_3$-complex has an essentially unique factorization 
as a connected sum of indecomposables, 
and these are either aspherical or have virtually free fundamental group.
There are many examples of the latter type which are not homotopy equivalent to 3-manifolds,
but the possible groups are largely known.
However the question of whether every aspherical $PD_3$-complex is homotopy
equivalent to a 3-manifold remains open.

We shall outline the work which lead to this reduction to the aspherical case, 
mention briefly remaining problems in connection with indecomposable 
virtually free fundamental groups,
and consider how we might show that $PD_3$-groups are 3-manifold groups.
We then state a number of open questions on 3-dimensional Poincar\'e duality 
groups and their subgroups, motivated by considerations from 3-manifold topology.
\end{abstract}

\keywords{Poincar\'e duality, $PD_3$-complex, $PD_3$-group}

\subjclass{57M25}

\maketitle
The first half of this article corresponds to my talk at the Luminy conference 
{\it Structure of 3-manifold groups\/} (26 February -- 2 March, 2018). 
When I was first contacted about the conference, 
it was suggested that I should give an expository talk on $PD_3$-groups and related open problems.
Wall gave a comprehensive survey of {\it Poincar\'e duality in dimension 3\/}
at the CassonFest in 2004, in which he considered the splitting of $PD_3$-complexes 
as connected sums of aspherical complexes and complexes with virtually free fundamental group,
and the JSJ decomposition of $PD_3$-groups along $\mathbb{Z}^2$ subgroups.
I have concentrated on the work done since then, 
mostly on $PD_3$-complexes with virtually free fundamental group, 
before considering possible approaches to showing that  aspherical $PD_3$-complexes 
might all be homotopy equivalent to closed 3-manifolds.

The second half is an annotated list of questions about $PD_3$-groups and their subgroups,
with relevant supporting evidence,
mostly deriving from known results for 3-manifold groups.
This began as an aide-memoire thirty years ago,
and was put on the arXiv in 2016.
The questions are straightforward, but  have largely resisted answers,
and  suggest the limitations of our present understanding.

\section{poincar\'e duality complexes}

Poincar\'e duality complexes were introduced by Wall to
model the homotopy types of closed manifolds \cite{wall}.

Let $X$ be a connected finite CW complex with fundamental group $\pi$,
and let $\Lambda=\mathbb{Z}[\pi]$ be the integral group ring, with its canonical involution
determined by inversion: $g\mapsto{g^{-1}}$.
Let $C_*(X;\Lambda)$ be the cellular chain complex of the universal cover $\widetilde{X}$,
considered as a complex of finitely generated free left modules over $\Lambda$.
Taking the $\Lambda$-linear dual modules and using the canonical involution of $\Lambda$
to swap right and left $\Lambda$-module structures gives a cochain complex $C^*(X;\Lambda)$ with 
\[
C^q(X;\Lambda)=\overline{Hom_\Lambda(C_q(X;\Lambda),\Lambda)}.
\]
Then $X$ is an {\it orientable (finite) $PD_n$-complex\/} if there is 
a fundamental class $[X]\in{H_n(X;\mathbb{Z})}$ such that slant product with a locally-finite 3-chain
in $C_3(\widetilde{X};\mathbb{Z})$ with image $[X]$ induces a chain homotopy equivalence  
\[
-\cap[X]:C^{n-*}(X;\Lambda)\to{C_*(X;\Lambda)}.
\]
This definition can be elaborated in various ways, firstly to allow for non-orientable analogues,
secondly to consider pairs (corresponding to manifolds with boundary),
and thirdly to weaken the finiteness conditions.
We shall focus on the orientable case, for simplicity, but $PD_3$-pairs arise naturally 
even when the primary interest is in the absolute case.
(Examples of the third type may arise as infinite cyclic covers of manifolds.)

Wall showed that every $PD_n$-complex is homotopy equivalent to an $n$-dimensional complex,
and in all dimensions {\it except\/} $n=3$ we may assume that there is a single $n$-cell.
Moreover, this top cell is essentially unique, and so there is a well-defined connected sum, 
for oriented $PD_n$-complexes.
When $n=3$  we may write $X$ as a union $X=X'\cup{e^3}$, where $c.d.X'\leq2$.
(Thus the exceptional case relates to the Eilenberg-Ganea Conjecture.)

Closed PL $n$-manifolds are finite $PD_n$-complexes, but there are simply
connected $PD_n$-complexes which are not homotopy equivalent to manifolds, 
in every dimension $n\geq4$.

\section{poincar\'e duality groups}

The notion of Poincar\'e duality group of dimension $n$ (or $PD_n $-group, 
for short) is an algebraic analogue of the notion of aspherical $n$-manifold. 

A finitely presentable group $G$ is a {\it $PD_n$-group\/} in the sense of Johnson and Wall if 
$K(G,1)$ is homotopy equivalent to a $PD_n$-complex \cite{jw72}.
Bieri and Eckmann gave an alternative purely algebraic formulation:
a group $G$ is a $PD_n$-group if the augmentation $\mathbb{Z}[G]$-module $\mathbb{Z}$ has a finite
projective resolution, $c.d.G=n$, $H^i(G;\mathbb{Z}[G])=0$ for $i<n$ 
and $H^n(G;\mathbb{Z}[G])$ is infinite cyclic as an abelian group \cite{be78}.
The right action of $G$ on this group determines the orientation character
$w_1(G):G\to\mathbb{Z}^\times$.
The group $G$ is {\it orientable\/} if  $H^n(G;\mathbb{Z}[G]$ is the augmentation module
(i.e., if $w_1(G)$ is the trivial homomorphism).

$PD_n$-groups are $FP$, and so are finitely generated, 
but there examples which are not finitely presentable, in every dimension $n\geq4$ \cite{da00}.
Whether there are $PD_3$-groups which are not finitely presentable remains unknown.
(The case $n=3$ is critical; there are examples of  $PD_n$-groups with 
all sorts of bad behaviour when $n>3$.
See \cite{da00} and the references there.)

It is still an open question whether every finitely presentable $PD_n$-group is the fundamental
group of a closed $n$-manifold.
(This is one aspect of the circle of ideas around the Novikov Conjecture.)

[If we define a {\it $PD_n$-space\/} to be a space homotopy equivalent to a $CW$-complex $X$ 
and such that $C_*(X;\Lambda)$ is chain homotopy equivalent to a finite complex of  
finitely generated {\it projective\/} left $\mathbb{Z}[\pi_1(X)]$-modules 
and with a class $[X]$ defining a duality chain homotopy equivalence as for $PD_n$-complexes,
then $G$ is a $PD_n$-group in the sense of Bieri and Eckmann if and only if 
$K(G,1)$ is an (aspherical) $PD_n$-space.
A $PD_n$-space $X$ is finitely dominated if and only if $\pi_1(X)$
is finitely presentable \cite{browd}.]

There is also a relative notion, of $PD_n$-pair of groups, which corresponds to an aspherical
compact manifold with incompressible, aspherical boundary components.
This notion arises naturally in connection with JSJ-decompositions of $PD_3$-groups.

Bieri  and Eckmann defined a $PD_n$-pair of groups $(G,\mathcal{S})$ 
as a group $G$  of finite cohomological dimension with a 
finite family $\mathcal{S}$ of monomorphisms $j_S:S\to{G}$, 
satisfying a modified form of Poincar\'e-Lefshetz duality.
The ``boundary" subgroups $j_S(S)$ are $PD_{n-1}$-groups, 
and $c.d.G=n-1$ if $\mathcal{S}$ is nonempty.
(There is another formulation, 
in terms of a group $G$ and a $G$-set $W$.
In this formulation the boundary subgroups arise as stabilizers
of points of $W$.
See \cite[page 138]{dd}.)

\section{low dimensions}

When $n=1$ or 2 the modelling of $n$-manifolds by $PD_n$-complexes
is precise: the only such complexes are homotopy equivalent to
the circle or to a closed surface, and two such manifolds 
are homeomorphic if and only if their groups are isomorphic. 

It is easy to see that a $PD_1$-complex $X$ must be aspherical,
and $\pi=\pi_1(X)$ has two ends and $c.d.\pi=1$.
Since $\pi$  is free of finite rank $r>0$ and $H_1(\pi;\mathbb{Z})$ is cyclic 
(or since $\pi$ is torsion-free and has two ends), 
$\pi\cong\mathbb{Z}$ and $X\simeq{S^1}$.
(There are elementary arguments do not require
cohomological characterizations of free groups
or of the number of ends.)

Every $PD_2$-complex with finite fundamental group is homotopy equivalent 
to either $S^2$ or the real projective plane $\mathbb{RP}^2$.
All others are aspherical.
Eckmann and M\"uller  showed that  every $PD_2$-complex with $\chi(X)\leq0$
is  homotopy equivalent to a closed surface, 
by first proving the corresponding result for $PD_2$-pairs  with nonempty boundary
and then showing that every $PD_2$-group splits over a copy of $\mathbb{Z}$ \cite{em80}.
Shortly afterwards, Eckmann and Linnell showed that there is no aspherical
$PD_2$-complex $X$ with $\chi(X)>0$ \cite{el81}.
Much later  Bowditch used ideas from geometric group theory to prove 
the stronger result that if $G$ is a finitely generated group such that 
$H^2(G;\mathbb{F}[G])$ has an $\mathbb{F}[G]$-submodule
of finite dimension over $\mathbb{F}$, for some field $\mathbb{F}$,
 then $G$ is commensurable with a surface group 
(i.e., the fundamental group of an aspherical closed surface) \cite{bow}.
One might hope for a topological argument,
based on improving a degree-1 map $f:M\to{X}$ with domain a closed surface.

The first non-manifold example occurs in dimension $n=3$.
Swan showed that every finite group of cohomological period 4 
acts freely on a finite-dimensional cell complex homotopy equivalent to $S^3$ \cite{swan}. 
The quotient complexes are $PD_3$-complexes.
(Swan's result predates the notion of $PD$-complex!)
However, if the group has non-central elements of order 2,
it cannot act freely on $S^3$ \cite{mil},
and so is not a 3-manifold group.
In particular, the symmetric group $S_3$ has cohomological period 4, 
but is not a 3-manifold group.
(By the much later work of Perelman, 
the finite 3-manifold groups are the fixed-point free finite subgroups of $SO(4)$.)

In these low dimensions $n\leq3$ it suffices to show that there 
is some chain homotopy equivalence
$C^{n-*}(X;\Lambda)\simeq{C_*(X;\Lambda)}$; 
that it is given by cap product with a fundamental class follows.

The {\it fundamental triple} of a $PD_3$-complex $X$ is $(\pi,w,c_{X*}[X])$,
where $\pi=\pi_1(X)$, $w=w_1(X)$ is the orientation character,
$c_X:X\to{K(\pi,1)}$ is the classifying map and $[X]$ is the fundamental class
in $H_3(X;\mathbb{Z}^w)$.
There is an obvious notion of isomorphism for such triples.
(Note, however, that in the non-orientable case it is only meaningful 
to specify the sign of $[X]$ if we work with pointed spaces.)
Hendriks showed that this is a complete homotopy invariant for such complexes.

\begin{thm}
\cite{he77}
Two $PD_3 $-complexes are (orientably) homotopy equivalent if and only 
if their fundamental triples are isomorphic.
\qed
\end{thm}    

Turaev has characterized the possible triples corresponding to 
a given finitely presentable group and orientation character
(the ``Realization Theorem").
In particular, he gave the following criterion.

\begin{thm}
\cite{tur}
A finitely presentable group $\pi$ is the fundamental group of an orientable finite $PD_3$-complex 
if and only if $I_\pi\oplus\Lambda^r\cong{J_\pi}\oplus\Lambda^s$ for some $r,s\geq0$, 
where $I_\pi$ is the augmentation ideal of $\pi$,
with finite rectangular presentation matrix $M$, and $J_\pi=\mathrm{Coker}(\overline{M})$.
\qed
\end{thm}    

C.B.Thomas gave an alternative set of invariants,
for orientable 3-manifolds, based on the Postnikov approach \cite{thom67}.
(The present formulation was introduced by Swarup, for orientable
3-manifolds, in \cite{swar74}.)

When $\pi$ is finite, $X$ is orientable and $\widetilde{X}\simeq{S^3}$,
and $X$ is determined by $\pi$ and the first nontrivial $k$-invariant $\kappa_2(X)\in{H^4(\pi;\mathbb{Z})}$.
Let $\beta:H^3(\pi;\mathbb{Q}/\mathbb{Z})\cong{H^4(\pi;\mathbb{Z})}$ be the Bockstein isomorphism.
Then $c_{X*}[X]$ and $\kappa_2(X)$  generate isomorphic cyclic groups,
and are paired by the equation $\beta^{-1}(\kappa_2(X))(c_{X*}[X])=\frac1{|\pi|}$.

When $\pi$ is infinite, $\pi_2(X)\cong\overline{H^1(\pi;\mathbb{Z}[\pi])}$
and $X$ is determined by the triple $(\pi,w,\kappa_1(X))$,
where $\kappa_1(X)\in{H^3(\pi;\pi_2(X))}$ is now the first nontrivial $k$-invariant.
In this case the connection between the two sets of invariants is not so clear.

The work of Turaev has been extended to the case of $PD_3$-pairs with aspherical boundary components
by Bleile \cite{bl}.
The relative version of the Realization Theorem proven there requires also 
that the boundary components be $\pi_1$-injective.
The Loop Theorem of Crisp \cite{cr2} should also be noted here.

The  homotopy type of a higher dimensional $PD_n$-complex $X$ is determined by
the triple $(P_{n-2}(X),w,f_{X*}[X])$,
where $f_X:X\to{P_{n-2}(X)}$ is the Postnikov $(n-2)$-stage  and 
$w=w_1(X)$ \cite{BB08}.
If $\widetilde{X}$ is $(n-2)$-connected then $P_{n-2}(X)\simeq{K(\pi,1)}$,
so this triple is a direct analogue of Hendriks' invariant.

\section{reduction to indecomposables}

In his foundational 1967 paper Wall asked whether $PD_3$-complexes
behaved like 3-manifolds with regards to connected sum  \cite{wall}.
Consider the following conditions
\begin{enumerate}
\item$X$ is a non-trivial connected sum;

\item$\pi=\pi_1(X)$ is a nontrivial free product;

\item either $\pi$ has infinitely many ends, or $\pi\cong{D_\infty}=Z/2Z*Z/2Z$.
\end{enumerate}
Clearly $(1)\Rightarrow(2)\Rightarrow(3)$.
Wall asked whether either of these implications could be reversed.
Turaev used his Realization Theorem to show that $(1)\Leftrightarrow(2)$ 
(the ``Splitting Theorem").

\begin{thm}
\cite{tur}
A $PD_3 $-complex $X$ is indecomposable with respect to connected sum 
if and only if $\pi=\pi_1(X)$  is indecomposable with respect to free product.
\qed
\end{thm}

The further analysis (for $\pi$ infinite) is based on the following three observations.

\begin{enumerate}
\item$\pi_2(X)\cong\overline{H^1(\pi;\Lambda)}$, by Poincar\'e duality;

\item{ Since $\pi$ is $FP_2$ we have
$\pi\cong\pi\mathcal{G}$, where $(\mathcal{G},\Gamma)$ is
a finite graph of finitely generated groups in which 
each vertex group has at most one end and each edge group is finite, 
by Theorem VI.6.3 of \cite{dd}.
Hence $H^1(\pi;\Lambda)$ has a ``Chiswell-Mayer-Vietoris" presentation
as a right $\Lambda$-module:
\begin{equation*}
\begin{CD}
0\to\oplus_{v\in{V_f}}\mathbb{Z}[G_v\backslash\pi]@>\Delta>>
\oplus_{e\in{E}}\mathbb{Z}[G_e\backslash\pi]\to{H^1(\pi;\Lambda)}\to0.
\end{CD}
\end{equation*}
Here $V_f$ is the set of vertices with finite vertex groups, 
$E$ is the set of edges, and the image of a coset $G_vg$ in $\pi$
under $\Delta$ is
\[
\Delta(G_vg)=\Sigma_{o(e)=v}(\Sigma_{G_eh\subset{G_v}}G_ehg)-
\Sigma_{t(e)=v}(\Sigma_{G_eh\subset{G_v}}G_ehg),
\]
where the outer sums are over edges $e$ and the inner sums are
over cosets of $G_e$ in $G_v$.
}
\item{Since $\widetilde{X}$ has only one nontrivial homology 
group in positive degrees,
\[H_i(C;\pi_2(X))\cong{H_{i+3}(C;\mathbb{Z})},
\] 
for any subgroup $C$ of $\pi$
and all $i\geq1$,
by a simple devissage.}
\end{enumerate}

These were first used together to show that
if $\pi$ is infinite and has a nontrivial finite normal subgroup 
then $X\simeq{S^1}\times\mathbb{RP}^2$ \cite{hi93}.

Crisp added an ingenious combinatorial argument
to give a substantial partial answer to the second part of Wall's question.

\begin{thm}
\cite{crisp}
Let $X$ be an indecomposable orientable $PD_3$-complex.
If $\pi=\pi_1(X)$ is not virtually free then $X$ is aspherical.
\qed
\end{thm}

His arguments strengthened the result on finite normal subgroups.

\begin{thm}
If $g\not=1\in\pi$ has finite order then either $w(g)=1$ and the centralizer $C_\pi(g)$ is finite
or $g^2=1$, $w(g)=-1$ and $C_\pi(g)$ has two ends. \qed
\end{thm}

If $C_\pi(g)$ has two ends it is in fact $\langle{g}\rangle\times\mathbb{Z}$,
by Corollary 7 of  \cite{hi17}.

The Realization Theorem (in the form given earlier), Crisp's result on centralizers 
and the ``Normalizer Condition" (the fact that proper subgroups of finite nilpotent groups
are properly contained in their normalizers) lead to 
an almost complete characterization of the class of indecomposable, 
virtually free groups which are fundamental groups
of orientable $PD_3$-complexes.

\begin{thm}
\cite{hi12}
If a finitely generated virtually free group $\pi$ is the fundamental group of 
an indecomposable orientable $PD_3$-complex $X$ then
$\pi=\pi\mathcal{G}$ where $(\mathcal{G},\Gamma)$ is a finite graph of finite groups
such that
\begin{enumerate}
\item{the underlying graph $\Gamma$ is a linear tree};
\item{all vertex groups have cohomological period dividing $4$,
and at most one is not dihedral};
\item{all edge groups are nontrivial, and at most one has order $>2$}.
\end{enumerate}
If an edge group has order $>2$ then it is $\mathbb{Z}/6\mathbb{Z}$,
and one of the adjacent vertex groups is $B\times{Z/dZ}$ with $B=T_1^*$ or $I^*$.
Every such group $\pi$ with all edge groups  $\mathbb{Z}/2\mathbb{Z}$
is the fundamental group of such a complex.
\qed
\end{thm}

The Realization Theorem is used to exclude subgroups of period $>4$,
as well as to support the final assertion.
The first infinite example \cite{hi04} is the group $S_3*_{Z/2Z}S_3$ with presentation
\[
\langle{a,b,c}\mid{cac=a^2,~cbc=b^2,~c^2=1}\rangle.
\]

We shall give an indication of the style of arguments in the following lemma.

\begin{lem} Let $X$ and $\pi$ be as in the theorem.
If $S\leq\pi$ is a $p$-group, for some prime $p$,
 then $S$ is cyclic or $S\cong{Q(2^k)}$ for some $k\geq3$.
\end{lem}

\begin{proof}
Since $\pi$ is virtually free it has a free normal subgroup $F$ of finite index.
Since $FS$ has finite index in $\pi$,  
the corresponding covering space $X_{FS}$ 
is again an orientable $PD_3$-complex.
After replacing $X$ by an indecomposable factor of $X_{FS}$,
if necessary, we may assume that $\pi=\pi\mathcal{G}$,
where $(\mathcal{G},\Gamma)$ is a finite graph of finite $p$-groups.
We may also assume that if an edge $e$ has distinct vertices 
$v,w$ then $G_e\not=G_v$ or $G_w$.
But then $N_\pi(G_e)$ is infinite, by the Normalizer Condition
and basic facts about free products with amalgamation.
Since $G_e$ is finite, $C_\pi(G_e)$ is also infinite,
which contradicts Crisp's result on centralizers. 
Hence there is only one vertex.
Similarly, there are no edges, and so $\pi=S$ is finite.
Hence $\widetilde{X}\simeq{S^3}$ and $S$ is as described.
\end{proof}

Are there examples with edge group $Z/6Z$?

Are there examples which ``arise naturally",
perhaps as infinite cyclic covers of closed 4-manifolds?
This is so for the generalized quaternionic group $Q(24,13,1)$
and for certain other finite groups which have cohomological period 4 
but are not 3-manifold groups \cite{hm86}.
[Note also that if $M$ is a closed 4-manifold with $\chi(M)=0$ and 
$f:\pi_1(M)\to\mathbb{Z}$ is an epimorphism with finitely generated
kernel $\kappa$ then the associated infinite cyclic covering space $M_\kappa$
is a $PD_3$-space, by Theorem 4.5 of \cite{Hi}.]

In the non-orientable case we have the following result.

\begin{thm}
\cite{hi17}
Let $P$ be an indecomposable non-orientable $PD_3$-complex.
Then $\mathrm{Ker}(w_1(P))$ is torsion-free.
If it is free then it has rank $1$.
\qed
\end{thm}

In particular, if $\pi_1(P)$ is not virtually free then $\pi_1(P)\cong\pi\mathcal{G}$, 
where each vertex group of $(\mathcal{G},\Gamma)$ has one end, 
and each edge group has order 2.
The orientation-preserving subgroups of the vertex groups are then $PD_3$-groups.
Examples of such $PD_3$-complexes which are 3-manifolds can be assembled from
quotients of punctured aspherical 3-manifolds by free involutions.
However, it is not yet known whether an indecomposable $PD_3$-complex with orientation cover 
homotopy equivalent to a 3-manifold must be homotopy equivalent to a 3-manifold.
On the other hand, if $\pi_1(P)$ is virtually free then $P\simeq{RP^2\times{S^1}}$ or $S^2\tilde\times{S^1}$
\cite{hi12}.

The arguments of this section apply with little change to the study of $PD_n$-complexes
with $(n-2)$-connected universal cover.
When $n$ is odd, the results are also similar.
However when $n$ is even they are in one sense weaker, in that it is not known whether the
group must be virtually torsion-free, and another sense stronger,
in that if $\pi$ is indecomposable, virtually free and has no dihedral subgroups
of order $>2$ then either $\pi$ has order $\leq2$
or it has two ends and its maximal finite subgroup have cohomology
of period dividing $n$.
(See \cite{bbh}.)

\section{aspherical case}

The work of Perelman implies that every homotopy equivalence between 
aspherical 3-manifolds is homotopic to a homeomorphism. 
It is natural to ask also whether every $PD_3 $-group 
is the fundamental group of an aspherical closed 3-manifold. 
An affirmative answer in general would suggest that a large part 
of the study of 3-manifolds may be reduced to algebra. 

If $G$ is a $PD_3 $-group which has a subgroup isomorphic to $\pi_1 (M)$ 
where $M$ is an aspherical 3-manifold then $G$ is itself a 3-manifold group.
For $M$ is either Haken, Seifert fibred or hyperbolic,
by the Geometrization Theorem of Perelman and Thurston,
and so we may apply \cite{zi82}, 
Section 63 of \cite{Z} or Mostow rigidity, respectively.
Thus it is no loss of generality to assume that $G$ is orientable.

It may also be convenient to assume also that $G$ is coherent,
and, in particular, finitely presentable. 
(No $PD_3$-group has ${F(2)\times{F(2)}}$ as a subgroup, 
which is some evidence that $PD_3$-groups might be coherent \cite{kr89}.)

The key approaches to this question seem to be through
\begin{enumerate}
\item{splitting over proper subgroups -- geometric group theory}; or

\item{homological algebra}; or

\item{topology}.
\end{enumerate}
Of course, there are overlaps between these.
The fact that $PD_2$-groups are surface groups is one common ingredient.

\smallskip
\noindent(1).
This approach has been most studied, 
particularly in the form of the Cannon Conjecture,
and there is a good exposition based on JSJ decompositions of (finitely presentable) 
$PD_3$-groups and pairs of groups (as in \cite{ds00})
in \cite{wall04}.
If one takes this approach it is natural to consider 
also the question of realizing $PD_3$-pairs of groups.

Splitting of $PD_3$-groups over proper subgroups was first considered by Thomas \cite{thom84}.
Kropholler showed that $PD_n$-groups with {\it {Max-c}}
(the maximum condition on centralizers) have canonical 
splittings along codimension-1 poly-$Z$ subgroups \cite{Kr90}.
(When $n=3$ such subgroups are $\mathbb{Z}^2$
or the Klein bottle group $\mathbb{Z}\rtimes\mathbb{Z}$.)
Castel showed that all $PD_3$-groups have {\it Max-c},
and used \cite{ss03} to give a JSJ decomposition for arbitrary $PD_3$-groups
(i.e., not assuming finite presentability) \cite{ca07}.

[Kropholler and Roller have considered splittings of 
a $PD_n$-group $G$ over subgroups which are $PD_{n-1}$-groups.
If $S$ is such a subgroup let  $\overline{\mathbb{F}_2[S]}=Hom(\mathbb{F}_2[S],\mathbb{F}_2)$.
Then Poincar\'e duality 
(for each of $G$ and $S$) and Shapiro's Lemma together
give $H^1(G;\overline{\mathbb{F}_2[S]}\otimes_{\mathbb{F}_2[S]}\mathbb{F}_2[G])=\mathbb{F}_2$,
and $G$ splits over a subgroup commensurable with $S$ if and only if the restriction to
$H^1(S;\overline{\mathbb{F}_2[S]}\otimes_{\mathbb{F}_2[S]}\mathbb{F}_2[G])$ is 0.
See \cite{[KR88],[KR88a],kr89}.]

In the simplest cases, $G$ is either solvable, of Seifert type or atoroidal 
(i.e., has no abelian subgroup of rank $>1$).
The solvable case is easy,
and the Seifert case was settled by Bowditch \cite{bow}.
(If $G/G'$ is infinite, this case follows from the earlier work of Eckmann, M\"uller and Linnell \cite{hi85}.)
The most studied aspect of the atoroidal case is the {\bf Cannon Conjecture},
that an atoroidal, Gromov hyperbolic $PD_3$-group 
should be a cocompact lattice in $PSL(2,\mathbb{C})$.
(In \cite{bm91} it is shown that a Gromov hyperbolic $PD_3$-group has 
boundary $S^2$, and in \cite{kk07} it is shown that an atoroidal $PD_3$-group 
which acts geometrically on a locally compact $CAT(0)$ space is Gromov hyperbolic.)

\smallskip
\noindent(2).
The homological approach perhaps has the least prospect of success, 
as it starts from the bare definition of a $PD_3$-group,
and needs something else, to connect with topology.
However it has proven useful in the subsidiary task of finding 
purely algebraic proofs for algebraic properties of 3-manifold groups, 
an activity that I have pursued for some time.
One can also show that if $G$ has sufficiently nice subgroups then it is a 3-manifold group.
For instance, if $G$ has a nontrivial $FP_2$ ascendant subgroup of infinite index,
then either $G$ is the group of a Seifert fibred 3-manifold or it is virtually the group
of a mapping torus \cite{hi06}.
In particular, $G$ is the fundamental group of a $\mathbb{S}ol^3$-manifold or 
of a Seifert fibred 3-manifold if and only if it has a nontrivial locally-nilpotent normal subgroup.

This strategy seems to work best when there subgroups which are surface groups.
One relatively new ingredient is  the Algebraic Core Theorem of Kapovich
and Kleiner \cite{kk05}, which ensures that this is so if $G$ has an $FP_2$ subgroup 
with one end.
If $H$ is a surface subgroup then either $H$ has finite index in its commensurator
$Comm_G(H)$ or $H$ has a subgroup $K$ of finite index such that $[G:N_G(K)]$ is finite
(and then $Comm_G(H)=Comm_G(K)\geq{N_G(K)}$ has finite index in $G$).
More generally,
if $H$ is a $FP_2$ subgroup then either $G$ is virtually the group of a mapping torus 
or $H$ has finite index in a subgroup $\widehat{H}$ which is its own normalizer in $G$.
Does $G$ then split over $H$?

However it remains possible that there may be $PD_3$-groups which are simple groups,
or even Tarski monsters, whose only proper subgroups are infinite cyclic.
It is then not at all clear what to do.
[Once again, the Davis construction may be used to give $PD_n$-groups
containing Tarski monsters, for all $n>3$ \cite{sa11}.]

\smallskip
\noindent(3).
To the best of my knowledge, no-one has explored the third option in any detail.
It has the advantage of direction connection with topology, 
but needs $G$ to be finitely presentable.
Here one starts from the fact that if $X$ is an orientable $PD_3$-complex
then there is a degree-1 map $f:M\to{X}$ with domain a closed orientable 3-manifold.
(This is not hard to see, but is a consequence of working in a low dimension.)
Since $\pi_1(f)$ is surjective and $\pi_1(X)$ is finitely presentable, 
$\mathrm{Ker}(\pi_1(f))$ is normally generated by finitely many elements of $\pi_1(M)$,
which may be represented by a link $L\subset{M}$.
(The link $L$ is far from being unique!)
We might hope to modify $M$ by Dehn surgery on $L$ to render the kernel trivial.
This {\it is\/} possible if $X$ is homotopy equivalent to a closed orientable 3-manifold $N$, 
for $M$ may then be obtained from $N$ by Dehn surgery on a 
link whose components are null-homotopic in $N$ \cite{Gad}. 
However, Gadgil's argument appears to use the topology of the target space 
in an essential way.
Moreover, there are $PD_3$-complexes which are not homotopy equivalent to manifolds,
and so this cannot be carried through in general.
(All known counterexamples have finite covers which are homotopy equivalent to manifolds.)

Let $L=\amalg_{i\leq{m}}L_i$ be a link in a 3-manifold $M$ and let
$n(L)=\amalg_{i\leq{m}}n(L_i)$ be an open regular neighbourhood of $L$ in $M$.
We shall say that $L$ admits a {\it drastic\/} surgery 
if there is a family of slopes
$\gamma_i\subset\partial{n(L_i)}$ such that the normal closure of
$\{[\gamma_1],\dots,[\gamma_n]\}$ in $\pi_1(M-n(L))$
meets the image of each peripheral subgroup $\pi_1(\partial{n(L_i)})$
in a subgroup of finite index.

If $X$ is an aspherical orientable $PD_3$-complex and $f:M\to{X}$ is a degree-1 map 
such that $\mathrm{Ker}(f_*)$ is represented by a link $L$ which admits a drastic surgery 
then after the surgery we may assume that $\mathrm{Ker}(f_*)$ is normally generated 
by finitely many elements of finite order.
Let $M=\#_{i=1}^rM_i$ be the decomposition into irreducibles.
Since $X$ is aspherical the map $f$ extends to a map from $f_\vee:\vee_{i=1}^rM_i\to{X}$.
Elementary considerations then show that $f_\vee$ restricts to a homotopy equivalence 
from one of the aspherical summands of $M$ to $X$.

Unfortunately there are knots which do not admit drastic surgeries,
but we do have considerable latitude in our choice of link 
$L$ representing $\mathrm{Ker}(f_*)$.
In particular, we may modify $L$ by a link homotopy, 
and so the key question may be:

\smallskip
\centerline{ 
is every knot $K\subset{M}$ {\it homotopic\/} to one 
which admits a drastic surgery?}

\smallskip
\noindent
The existence of $PD_3$-complexes which are not homotopy equivalent to
3-manifolds shows that we cannot expect a stronger result,
in which 
``meets the image $\dots$ finite index" is replaced by 
``contains $\dots$ $\pi_1(\partial\overline{n(L_i)})$"
in the definition of drastic surgery.
Can we combine Dehn surgery with passage to finite covers and
varying $L$ by link-homotopy?

[There is a parallel issue in the $PD_2$-case. 
Here the strategy can be justified {\it ex post facto\/}: 
if  a degree-1 map $f:M\to{N}$ of closed orientable surfaces is not a homotopy equivalence
then there is a non-separating simple closed curve $\gamma\subset{M}$ 
with image in the kernel of $\pi_1(f)$ \cite{Ed}.
Surgery on $\gamma$ replaces $f$ by a new degree-1 map $f':M'\to{N}$, 
where $\chi(M')=\chi(M)+2$.
After finitely many iterations we obtain a 
degree-1 map $\hat{f}:\hat{M}\to{N}$, with $\chi(\hat{M})=\chi(N)$. 
Such a map must be a homotopy equivalence.
However it seems that Edmonds' argument requires the codomain $N$ to also be a 2-manifold,
which is what we want to prove!
Can we avoid a vicious circle?]

\smallskip
On a more speculative level, can we use stabilization with products to bring 
the methods of high dimensional topology (as in the Novikov Conjecture) to bear?
Is $G\times\mathbb{Z}^r$ realizable by an aspherical $(r+3)$-manifold for some $r>0$?

Among the most promising new ideas for studying $PD_3$-groups
since Wall's survey are the JSJ decomposition for arbitrary $PD_3$-groups \cite{ca07},
based on the work of Scott and Swarup,  the Algebraic Core Theorem,
based on coarse geometry \cite{kk05}, and the use of profinite and pro-$p$ completions, 
particular in connection with the Tits alternative, as in \cite {bb17} and \cite {kz07}.

\section{questions on $PD_3$-groups and their subgroups: preamble}

In the following sections we shall present a number of questions on subgroups of $PD_3 $-groups, 
motivated by results conjectured or already established geometrically for 
3-manifold groups. 
The underlying question is whether every $PD_3$-group $G$
is the fundamental group of some aspherical closed 3-manifold,
and has been discussed above.
The following questions represent possibly simpler consequences.
(If we assume $G$ is coherent and has a finite $K(G,1)$-complex,
as is the case for all 3-manifold groups, a number of these questions have
clear answers.)

Prompted by the main result of \cite{kk05}, we define an {\it open}
$PD_n$-group to be a countable group $G$ of cohomological dimension 
$\leq n-1$ such that every  nontrivial $FP$ subgroup $H$ with
$H^s(H;\mathbb{Z}[H])=0$ for $s<n-1$ is the ambient group of a 
$PD_n$-pair $(H,\mathcal{T})$, for some set of monomorphisms $\mathcal{T}$.
Every subgroup of infinite index in a $PD_3$-group $G$ is an open $PD_3$-group 
in our sense, by Theorem 1.3 of \cite{kk05}.
(The analogies are precise if $n=2$, but these definitions are too broad
when $n\geq4$. We shall consider only the case $n=3$.)

The corresponding questions for subgroups of open $PD_3$-groups
should be considered with these. 
Any group with a finite 2-dimensional Eilenberg -- Mac Lane complex is the 
fundamental group of a compact aspherical 4-manifold with boundary, 
obtained by attaching 1- and 2-handles to $D^4$.
(Conjecturally such groups are exactly the finitely presentable
groups of cohomological dimension 2).
On applying the reflection group trick of Davis to the boundary we see that 
each such group embeds in a $PD_4$-group \cite{da00}. 
Thus the case considered here is critical.

We assume throughout that $G$ is an orientable $PD_3 $-group. 
The normalizer and centralizer of a subgroup $H$ of $G$ shall be denoted by 
$N_G (H)$ and $C_G (H)$, respectively. 
We shall also let $\zeta G=C_G(G)$, $G'$ and $G^{(\omega )}=\cap G^{(n)}$
denote the centre, the commutator subgroup and the intersection of the terms 
of the derived series of $G$, respectively. 
A group has a given property {\it virtually} if it
has a subgroup of finite index with that property.

Since we are interested in $PD_3$-groups,
we shall use {\it 3-manifold group\/} henceforth to mean 
{\it fundamental group of an aspherical closed 3-manifold}.

\section{the group}

If $M=K(G,1)$ is a closed 3-manifold we may assume it has
one 0-cell and one 3-cell, and equal numbers of 1- and 2-cells.
Hence $G$ has a finite presentation of deficiency 0;
this is clearly best possible, 
since $\beta_1(G;\mathbb{F}_2)=\beta_2(G;\mathbb{F}_2)$.
Moreover $G$ is $FF$,
i.e., the augmentation module $\mathbb{Z}$ has a finite {\it free} 
$\mathbb{Z}[G]$-resolution,
while $\tilde K_0 (\mathbb{Z}[G])=Wh(G)=0$ 
and $\tilde M\cong\mathbb{R}^3$, so $G$ is 1-connected at $\infty$.

In general, the augmentation $\mathbb{Z}[G]$-module $\mathbb{Z}$ 
has a finite projective resolution, 
so $G$ is almost finitely presentable ($FP_2$),
and there is a 3-dimensional $K(G,1)$ complex. 
The $K(G,1)$-complex is finitely dominated, and hence a Poincar\'e complex 
in the sense of \cite{wall}, if and only if $G$ is finitely presentable.

For each $g\in G$ with infinite conjugacy class 
$[G:C_G(\langle g\rangle)]=\infty$, 
so $c.d.C_G(\langle g\rangle)\leq2$  \cite{[St77]}.
Hence $C_G(\langle g\rangle)/\langle g\rangle$ is locally virtually free,
by Theorem 8.4 of  \cite{[B]}.
Therefore $G$ must satisfy the Strong Bass Conjecture, by  \cite{[Ec86]}.

An $FP_2$ group $G$ such that $H^2(G;\mathbb{Z}[G])\cong\mathbb{Z}$  
is virtually a $PD_2$-group  \cite{bow}.

\begin{enumerate}
\item Is $G$ finitely presentable? 

\item If $G$ is finitely presentable does it have deficiency 0?

\item Is $G$ of type $FF$?

\item Is $\tilde K_0 (\mathbb{Z}[G])=0$? Is $Wh(G)=0$?  

\item Is $G$ 1-connected at $\infty$? 

\item Is $K(G,1)$ homotopy equivalent to a finite complex?

\item If $G$ is an $FP_3$ group such that $H^3(G;\mathbb{Z}[G])\cong\mathbb{Z}$ is
$G$ virtually a $PD_3$-group?
\end{enumerate}

If (7) is true then centres of 2-knot groups are finitely generated.

\section{subgroups in general}

Since $G$ has cohomological dimension 3 it has no nontrivial finite subgroups. 
Any nontrivial element $g$ generates an infinite cyclic subgroup 
$\langle g\rangle$; it is not known 
whether there need be any other proper subgroups. 
If a subgroup $H$ of $G$ has finite index then it is also a $PD_3 $-group. 
The cases when $[G:H]$ is infinite are of more interest, and then either 
$c.d.H=2$ or $H$ is free, by  \cite{[St77]} and  \cite{[S]}. 
If there is a finitely generated (respectively, $FP_2$) subgroup of 
cohomological dimension 2 there is one such which has one end 
(i.e., which is indecomposable with respect to free product).   
A solvable subgroup $S$ of Hirsch length $h(S)\geq2$ must be finitely presentable, 
since either $[G:S]$ is finite or $c.d.S=2=h(S)$  \cite{[Gi79]}.
(In particular, abelian subgroups of rank $>1$ are finitely generated.)

3-manifold groups are {\it coherent}: 
finitely generated subgroups are finitely presentable. 
In fact something stronger is true: if $H$ is a finitely generated subgroup
it is the fundamental group of a compact 3-manifold (possibly with boundary)  \cite{[Sc73]}.
We shall say that a group $G$ is {\it almost coherent\/} if
every finitely generated subgroup of $G$ is $FP_2$.
This usually suffices for homological arguments,
and is implied by either coherence of the group or coherence of the group ring.
(If $\pi$ is the fundamental group of a graph manifold then
the group ring $\mathbb{Z}[\pi]$ is coherent.
The corresponding result for lattices in $PSL(2,\mathbb{C})$
is apparently not known.)

If $G$ is a $PD_3$-group with a one-ended $FP_2$ subgroup $H$ 
then there is a system of monomorphisms $\sigma $ such that $(H,\sigma)$ 
is a $PD_3 $-pair \cite{kk05}. 
Hence $\chi(H)\leq0$.
In particular, no $PD_3$-group has a subgroup $F\times{F}$ with $F$ a noncyclic free group.
(This was first proven in \cite{kr89}.)
As such groups $F\times F$ have finitely generated subgroups which are not 
finitely related (cf. Section 8.2 of \cite{[B]}),
this may be regarded as weak evidence for coherence.
(On the other hand, every surface group $\sigma$ with $\chi (\sigma )<0$ has such a  
subgroup $F$ and so $F\times F$ is a subgroup of $\sigma\times\sigma$.
Thus $PD_n $-groups with $n\geq 4$ need not be coherent.)

Let $M$ be a closed orientable 3-manifold.
Then $M$ is Haken, Seifert fibred or hyperbolic, 
by the Geometrization Theorem.
With  \cite{[KM10]} it follows that if $\pi_1(M)$ is infinite then 
it has a $PD_2$-subgroup.
A transversality argument implies that every element of 
$H_2(M;\mathbb{Z})\cong H^1(M;\mathbb{Z})\cong [M;S^1]$ 
is represented by an embedded submanifold.
If $M$ is aspherical it follows that $H_2 (\pi_1 (M);\mathbb{Z})$ is generated 
by elements represented by surface subgroups of $\pi_1 (M)$.

If $G/G'$ is infinite then $G$ is an HNN extension with finitely generated 
base and associated subgroups  \cite{[BS78]}, and so has a finitely generated 
subgroup of cohomological dimension 2. 
If, moreover, $G$ is almost coherent then it has a $PD_2$-subgroup \cite{kk05}.

If $H$ is a subgroup of $G$ which is a $PD_2$-group then $H$ has finite index in a maximal such subgroup.
This is clear if $\chi(H)<0$, by the multiplicativity of $\chi$ in the passage to subgroups of finite index.
If $\chi(H)=0$ we argue instead that an infinite increasing union of copies of $\mathbb{Z}^2$ must have
cohomological dimension 3.

\begin{enumerate}
\addtocounter{enumi}{7}
\item Is there a noncyclic proper subgroup? 
If so, is there one of cohomological dimension 2? and finitely generated?

\item Is there a subgroup which is a surface group? 

\item Is every element of $H_2(G;\mathbb{Z})$ represented by a $PD_2$ subgroup?

\item Is $G$ (almost) coherent? 

\item Is $\mathbb{Z}[G]$ coherent as a ring? 
        
\item Does every (finitely presentable) subgroup of cohomological 
dimension 2 have a (finite) 2-dimensional Eilenberg-Mac Lane complex 
(with $\chi\leq0$)?

\item Let $H$ be a finitely generated subgroup with one end and of infinite index 
in $G$. Does $H$ have infinite abelianization? contain a surface group?

\end{enumerate}

\section{ascendant subgroups}

If $M$ is a closed aspherical 3-manifold which is not a graph manifold
then $M$ has a finite covering space which fibres over the circle \cite{[Ag12], [PW12]}.
Hence indecomposable finitely generated subgroups of infinite index in such groups 
are (finitely presentable) semidirect products $F\rtimes\mathbb{Z}$, 
with $F$ a free group.
Such groups are HNN extensions with finitely generated free base, 
and associated subgroups free factors of the base \cite{[FH99]}.

If $N$ is an $FP_2$ ascendant subgroup of $G$ and 
$c.d.N=2$ then it is a surface group and $G$ has a subgroup of finite index 
which is a surface bundle group.
If $c.d.N=1$ then $N\cong\mathbb{Z}$ and either $G$ is virtually poly-$Z$ or 
$N$ is normal in $G$ and $[G:C_G (N)]\leq 2$  \cite{[BH91],hi06}. 
In the latter case $G$ is the group of a Seifert fibred 3-manifold  \cite{bow}. 
It is easy to find examples among normal subgroups of 3-manifold groups to show 
that finite generation of $N$ is necessary for these results.
                               
If $N$ is finitely generated, normal and $[G:N]=\infty$ then 
$H^1 (G/N;\mathbb{Z}[G/N])$ is isomorphic to $H^1 (G;\mathbb{Z}[G/N])$
and hence to $H_2 (G;\mathbb{Z}[G/N])\cong{H_2 (N;\mathbb{Z})}$,
by Poincar\'e duality. 
If $G/N$ has two ends, then
after passing to a subgroup of finite index,
we may assume that $G/N\cong\mathbb{Z}$. 
Shapiro's lemma and Poincar\'e duality (for each of $G$ and $G/N$)
together imply that 
$lim_\rightarrow{H^2 (N; M_i)}$ is 0 for any direct system $M_i$ with limit 0.
(See Theorem 1.19 of \cite{Hi}.)
Hence $N$ is $FP_2$ by Brown's criterion  \cite{[Br75]} and so is a surface group by
the above result.

\begin{enumerate}
\addtocounter{enumi}{14}

\item Is there a simple $PD_3 $-group?

\item Is $G$ virtually representable onto $\mathbb{Z}$?

\item Must a finitely generated {\it normal} subgroup $N$ be finitely 
presentable?

\item Suppose $N\leq U$ are subgroups of $G$ with $U$ finitely generated
and indecomposable, $[G:U]$ infinite, $N$ subnormal in $G$ and $N$ not cyclic. 
Is $[G:N_G(U)]<\infty$? (Cf.  \cite{[El84]}.)

\item{ Let $G$ be a $PD_3$-group such that $G'$ is free.
Is $G$ a semidirect product $K\rtimes\mathbb{Z}$ with $K$ a $PD_2$-group?}

\end{enumerate}

\section{centralizers, normalizers and commensurators}

If $G$ is a $PD_3$-group with nontrivial centre then 
$\zeta{G}$ is finitely generated and $G$ is the fundamental 
group of an aspherical Seifert fibred 3-manifold  \cite{bow}.
(See also  \cite{hi85}.) 
Since an elementary amenable group of finite cohomological dimension is 
virtually solvable  \cite{[HL92]},
it follows also that either $G$ is virtually poly-$Z$ or
its maximal elementary amenable normal subgroup is cyclic.

Every strictly increasing sequence of centralizers $C_0<C_1<\dots <C_n=G$ 
in a $PD_3$-group $G$ has length $n$ at most 4  \cite{hi06}.
(The finiteness of such sequences in any $PD_3$-group 
is due to Castel  \cite{ca07}.)
On the other hand,  the 1-relator group with presentation 
$\langle{t,x}\mid{tx^2t^{-1}=x^3}\rangle$
has an infinite chain of centralizers,
and hence so does the $PD_4$-group obtained 
from it by the Davis construction \cite{[Kr93]}.

If the sequence of centralizers $C_1\cong\mathbb{Z}<C_2<C_3<C_4=G$ is 
strictly increasing then $C_3$ must be nonabelian. 
(See  \cite{hi06}.)
Hence it is $FP_2$ \cite{ca07}, 
and so either $G$ is Seifert or $c.d.C_3=2$.
In all cases it follows that $C_2\cong\mathbb{Z}^2$.
Equivalently, if $G$ has a maximal abelian subgroup $A$
which is not finitely generated then $1<A<G$ is the only 
sequence of centralizers containing $A$.

If every abelian subgroup of $G$ is finitely generated then
the centralizer $C_G (x)$ of any $x\in{G}$ is finitely generated  \cite{ca07}.
It then follows that every centralizer is either $\mathbb{Z}$,
finitely generated and of cohomological dimension 2
or of index $\leq2$ in $G$  \cite{hi06}.
(Applying the Davis construction to the group with presentation 
$\langle{t,x}\mid{txt^{-1}=x^2}\rangle$ gives a $PD_4$-group 
with an abelian subgroup which is not finitely generated
\cite{[Me90]}.)

An element $g$ is a {\it root\/} of $x$ if $x=g^n$ for some $n$.
All roots of $x$ are in $C_G (x)$.
If $C_G(x)$ is finitely generated then $x$ is not infinitely divisible.
For if $c.d.C_G(x)=1$ then $C_G(x)\cong\mathbb{Z}$;
if $c.d.C_G(x)=2$ then $C_G(x)/\langle{x}\rangle$ is
virtually free, by Theorem 8.4 of  \cite{[B]};
and if $c.d.C_G(x)=3$ then $C_G(x)/\langle{x}\rangle$ is 
virtually a $PD_2$-group  \cite{bow}.
Conversely, if $x$ is not infinitely divisible then
$C_G(x)$ is finitely generated  \cite{ca07}.

If $C_G (x)$ is nonabelian then it is $FP_2$, 
and is either of bounded Seifert type or has finite index in $G$ \cite{ca07}.
In the latter case either $[G:C_G (x)]\leq2$ or $G$ is virtually $\mathbb{Z}^3$,
by Theorem 2 of  \cite{hi06}.

If $x$ is a nontrivial element of $G$ then 
${[N_G (\langle x\rangle):C_G (x)]}\leq 2$
(since $\langle{x}\rangle\cong\mathbb{Z}$).
If $F$ is a finitely generated nonabelian free subgroup of $G$ 
then $N_G(F)$ is finitely generated and $N_G(F)/F$ is finite 
or virtually $\mathbb{Z}$  \cite{hi06}.
(See  \cite{[Sc76]} for another argument in the 3-manifold case.)
If $H$ is an $FP_2$ subgroup which is a nontrivial free product 
but is not free then $[N_G(H):H]<\infty$ and $C_G(H)=1$  \cite{hi06}.

If $H$ is a one-ended $FP_2$ subgroup of infinite index in $G$ 
then either $[G:N_G(H)]$ or $[N_G(H):H]$ is finite.
(See Lemma 2.15 of  \cite{Hi}).
More precisely, 
define an increasing sequence of subgroups $\{ H_i |i\geq 0\}$ by 
$H_0 =H$ and $H_i =N_G (H_{i-1} )$ for $i>0$. 
Then $\hat H=\cup H_i $ is $FP_2$
and either $c.d.\hat H=2$, 
$\hat H$ has one end and $N_G (\hat H)=\hat H$, or $\hat H$ is a
$PD_3 $-group and $G$ is virtually the group of a surface bundle,
by Theorem 2.17 of  \cite{Hi}.
In particular, if $G$ has a subgroup $H$ which is a surface group with 
$\chi (H)=0$ (respectively, $<0$) then either it has such a subgroup 
which is its own normalizer 
in $G$ or $G$ is virtually the group of a surface bundle.

The {\it commensurator} in $G$ of a subgroup $H$ is the subgroup 
\[
Comm_G (H)=\{ g\in G\mid 
[H:H\cap gHg^{-1} ]<\infty ~and ~[H:H\cap g^{-1} Hg]<\infty\}.
\]
It clearly contains $N_G (H)$.

If $x\not=1$ in $G$ then the Baumslag-Solitar relation $tx^p t^{-1} =x^q $ 
implies that $p=\pm{q}$  \cite{ca07}. 
It follows easily that
$Comm_G (\langle x\rangle)=\cup N_G (\langle x^{n!}\rangle)$.
Since the chain of centralizers $C_G (\langle x^{n!}\rangle$ is increasing and 
$[N_G(\langle{x^k}\rangle:C_G(\langle{x^k}\rangle]\leq2$
for any $k$ it follows that
$Comm_G(\langle x\rangle)=N_G(\langle x^{n!}\rangle)$
for some $n\geq1$.

If $H$ is a $PD_2$-group then Theorem 1.3 and Proposition 4.4 of
 \cite{[KR89a]} imply that either $[Comm_G (H):H]<\infty$ 
or $H$ is commensurable with a subgroup $K$ such that $[G:N_G(K)]<\infty$,
and so $[G:Comm_G(H)]<\infty$.
This dichotomy is similar to the one for normalizers of $FP_2$ subgroups
cited above. 
It can be shown that if $H\cong\mathbb{Z}^2$ then either $Comm_G(H)=N_G(H)$
or $G$ is virtually $\mathbb{Z}^3$.
However, the exceptional cases do occur.
If $G=B_1$ is the flat 3-manifold group with presentation
$\langle{t,x,y}\mid{txt^{-1}=x^{-1},~ty=yt,~xy=yx}\rangle$ and
$A$ is the subgroup generated by $\{t,y\}$ then $N_G(A)=A$
but $Comm_G(A)=G$.

\begin{enumerate}
\addtocounter{enumi}{19}

\item Is every abelian subgroup of $G$ finitely generated? 

\item If $G$ is not virtually abelian and $H$ is an $FP_2$ subgroup such that $N_G(H)=H$ is $[Comm_G(H):H]$ finite?
\end{enumerate}

\section{the derived series and perfect subgroups}

Let $G^{(\omega)}=\cap{G^{(n)}}$ be the intersection of the terms of the derived series for $G$.
If $G^{(\omega)}=G^{(n)}$ for some finite $n$ then $n\leq3$, and $G/G^{(\omega)}$ 
is either a finite solvable group with cohomological period dividing 4, 
or has two ends and is $\mathbb{Z}$, $\mathbb{Z}\oplus{Z/2Z}$ or $D_\infty=Z/2Z*Z/2Z$,
or has one end and is a solvable $PD_3$-group.
(The argument given in \cite{[Hi96]} for orientable 3-manifold groups also applies here.)
There is a similar result for the lower central series.
If $G$ is orientable and $G_{[\omega]}=G_{[n]}$ for some finite $n$ then $n\leq3$,
and $G/G_{[\omega]}$ is finite, $\mathbb{Z}$  or a nilpotent $PD_3$-group \cite{[Te97]}.

If $G$ is not virtually representable onto $\mathbb{Z}$ then $G/G^{(\omega )} $ is
either a finite solvable group with cohomological period dividing 4
(and $G^{(\omega )} $ is a perfect $PD_3 $-group) or is a finitely generated,
infinite, residually finite-solvable group with one or infinitely many ends.
Let $M$ be the (aspherical) 3-manifold obtained by 
0-framed surgery on a nontrivial knot $K$ with Alexander polynomial $\Delta_K\dot=1$,
and let $G=\pi_1(M)$.
Then $G'$ is a perfect normal subgroup which is not finitely generated. 
(In this case $G_{[\omega]}=G^{(\omega)}=G'$, and $G/G^{(\omega)}\cong\mathbb{Z}$.)
Replacing a suitable solid torus in $RP^3\#RP^3$ by the exterior of such a knot $K$ gives an example with
$G/G^{(\omega)}\cong{D_\infty}$.

Let $\kappa$ be a perfect normal subgroup of the fundamental group $\pi$ of a $PD_3$-complex $X$.
Then $\rho=\pi/\kappa$ is $FP_2$, since $\pi$ is $FP_2$ and $H_1(\kappa;\mathbb{Z})=0$.
The arguments of  \cite{crisp} give $\rho\cong(*_{i=1}^rG_i)*V$,
where each factor $G_i$ has one end and $V$ is virtually free.
Moreover, if $\rho$ is infinite and has a nontrivial finite normal subgroup then $\rho$ has two ends.
(However, the further analysis of \cite{hi12} does not apply, since
there is no analogue of the Splitting Theorem of Turaev.)
We also have $H_2(\kappa;\mathbb{Z})\cong{H^1(\rho;\mathbb{Z}[\rho])}$ as an abelian group.
In particular, if $\kappa$ is acyclic then $\rho$ is a $PD_3$-group.

The intersection $P=\cap G^{(\alpha )} $ of the terms of the transfinite 
derived series for $G$ is the maximal perfect subgroup of $G$, 
and is normal in $G$.
The quotient $G/P$ is $FP_2$.
If $P\not=1$ and $[G:P]$ is infinite then $c.d.P=2$,
but $P$ cannot be $FP_2 $, 
for otherwise it would be a surface group \cite{[Hi87]}. 
Note that $P\subseteq G^{(\omega )} $, and if $c.d.P=2$ then 
$c.d.G^{(\omega )} =2$ also. 
If $[G:P]$ is infinite and $\zeta G\not= 1$ then $P=1$.

If $G$ is a $PD_3$-group and $H$ is a nontrivial $FP_2$ subgroup
such that $H^1(H;\mathbb{Z})=0$ then $[G:H]$ is finite.
(Use \cite{kk05}. See \cite{[Ho82]} for 3-manifold groups.)

\begin{enumerate}  
\addtocounter{enumi}{21}
      
\item Can a nontrivial finitely generated normal subgroup of infinite index be perfect?
acyclic?

\item If a finitely generated, infinite, residually solvable group has infinitely 
many ends must it be virtually representable onto $\mathbb{Z}$?

\item If $P=1$ is $G$ residually solvable (i.e., is $G^{(\omega )} =1$ also)?
\end{enumerate}

\section{the tits alternative}

A group {\it satisfies the Tits alternative\/} if every finitely generated subgroup is either solvable
or contains a non-abelian free group.

Let $N$ be the subgroup generated by all the normal subgroups which have 
no nonabelian free subgroup. 
Then $N$ is the maximal such subgroup, and clearly it contains the
maximal elementary amenable normal subgroup of $G$.
If $N$ is nontrivial then either $N\cong\mathbb{Z}$, $c.d.N=2$ or $N=G$.
If $N$ is a rank 1 abelian subgroup then $N\cong\mathbb{Z}$.
(For otherwise $N\leq G'$ and $G'\leq C_G(N)$, 
so either $[G:C_G(N)]$ is finite, which can be excluded by \cite{bow},
or $G'$ is abelian, by Theorem 8.8 of \cite{[B]},
in which case $G$ is solvable and hence virtually poly-$Z$,
and $N$ must again be finitely generated.)
If $c.d.N=2$ then $N$ cannot be $FP_2$, 
for otherwise it would be a surface group and $G$ 
would be virtually the group of a surface bundle \cite{[Hi87]}.
Since $N$ has no nonabelian free subgroup this would imply that $N$ 
and hence $G$ are virtually poly-$Z$, and so $N=G$. 
Similarly, if $N=G$ and $G/G'$ has rank at least 2 then 
there is an epimorphism $\phi:G\to\mathbb{Z}$ with finitely generated kernel \cite{[BNS87]}. 
Hence $\mathrm{Ker}(\phi)$ is a surface group and so $G$ is poly-$Z$. 

A finitely generated, torsion-free group is {\it properly  locally cyclic\/} 
if every finitely generated subgroup of infinite index is cyclic.
If $G$ is an almost coherent $PD_3$-group which is not virtually properly cyclic 
then every finitely generated subgroup of $G$ satisfies the Tits alternative \cite{bb17}.
(In fact it suffices for their argument for  ``almost coherent" to be assumed only for the subgroup, 
as in \cite{[Hi03]}).
We may then use \cite{[Gi79]} and Corollary 1.4 of \cite{kk05} to show that
solvable subgroups are abelian or virtually poly-$Z$.

\begin{enumerate}
\addtocounter{enumi}{24}

\item Is $N$ the maximal elementary amenable normal subgroup?

\item If $H$ is a finitely generated subgroup which has no nonabelian 
free subgroup must it be virtually poly-$Z$?

\item{In particular, is a $PD_3$-group of subexponential growth virtually nilpotent?}
\end{enumerate}

\section{atoroidal groups}

We shall say that $G$ is {\it atoroidal} if all of its finitely generated 
abelian subgroups are cyclic. 
Two-generator subgroups of atoroidal, almost coherent $PD_3$-groups
are either free or of finite index, by  \cite{BS89} together with the 
Algebraic Core Theorem of \cite{kk05}.
3-Manifolds with atoroidal fundamental group are hyperbolic, by the Geometrization Theorem.
Every closed hyperbolic 3-manifold has a finite covering space which fibres over the circle
\cite{[Ag12],[PW12]}.

If an atoroidal $PD_3$-group acts geometrically on a locally compact
$CAT(0)$ space then it is Gromov hyperbolic \cite{[KK07]}.
A Gromov hyperbolic $PD_3$-group has boundary $S^2$ \cite{bm91}.

\begin{enumerate} 
\addtocounter{enumi}{27}
      
\item{Is every atoroidal $PD_3 $-group Gromov hyperbolic?}

\item{Does every atoroidal $PD_3 $-group have a boundary in the sense of  \cite{[Be96]}?}

\item{The Cannon Conjecture: is every Gromov hyperbolic $PD_3 $-group
isomorphic to a discrete uniform subgroup of $PSL(2,\mathbb{C})$?}

\item{Does every atoroidal $PD_3 $-group have a nontrivial 
finitely generated subnormal subgroup of infinite index?}

\end{enumerate}

\section{splitting}

The central role played by incompressible surfaces in the geometric study 
of Haken 3-manifolds suggests strongly the importance of splitting theorems for $PD_3$-groups.
This issue was raised in \cite{thom84}, the first paper on $PD_3$-groups.
Kropholler and Roller considered splittings of $PD_n$-groups over $PD_{n-1}$ subgroups 
\cite{[KR88],[KR88a],kr89,[KR89a]}; see also \cite{[Du89]} and \cite{kn13}.
Kropholler gave two different formulations of a torus theorem for $PD_3$-groups,
one extending to higher dimensions but requiring the hypothesis that the group 
have  {\it Max}-$c$, the maximal condition on centralizers  \cite{Kr90},
and the other with a weaker conclusion \cite{[Kr93]}.
Castel has since shown that every $PD_3$-group has  {\it Max}-$c$,
and has given a JSJ-decomposition theorem for $PD_3$-groups and group pairs \cite{ca07}.

In particular, if $G$ has a subgroup $H\cong\mathbb{Z}^2 $ then either $G$ 
splits over a subgroup commensurate with $H$ or it
has a nontrivial abelian normal subgroup \cite{Kr90},
and so is a 3-manifold group \cite{bow}.
If $G$ splits over a $PD_2$-group $H$ then either $G$ is virtually a semidirect product
or $N_G(H)=H$. (See \S7 of \cite{hi06}.)

If $G$ is an ascending HNN extension with $FP_2$ base $H$ then $H$ is 
a $PD_2$-group and is normal in $G$,
and so $G$ is the group of a surface bundle. 
(This follows from Lemma 3.4 of \cite{[BG85]}.)
If $G$ has no noncyclic free subgroup and $G/G'$ is infinite then $G$ is an 
ascending HNN extension with finitely generated base and associated subgroups.
If $G$ is residually finite and has a subgroup isomorphic to $\mathbb{Z}^2$ then either 
$G$ is virtually poly-$Z$ or it has subgroups of finite index with 
abelianization of arbitrarily large rank. 
(A residually finite $PD_3$-group which has a subgroup $H\cong\mathbb{Z}^2$ is 
virtually split over a subgroup commensurate with $H$ \cite{[KR88]}, so 
we may suppose that $G$ splits over $\mathbb{Z}^2$,
and then we may use the argument of \cite{[Ko87]}, which is essentially algebraic.)

\begin{enumerate}
\addtocounter{enumi}{31}

\item{If $G$ is a nontrivial free product with amalgamation or 
HNN extension does it split over a $PD_2$ group?}
  
\item{If $G$ is a nontrivial free product with amalgamation 
is it virtually representable onto $\mathbb{Z}$?}

\item  Can $G$ be a properly ascending HNN extension 
(with base not $FP_2$)?

\item{If $G$ has a subgroup $H$ which is a $PD_2$-group and such that $N_G(H)=H$
does $G$ have a subgroup of finite index which splits over $H$?
In particular is this so if $H\cong\mathbb{Z}^2$?}

\item{Suppose $G$ is not virtually poly-$Z$ and that $G/G'$ is infinite. 
Does $G$ have subgroups of finite index whose abelianization has rank $\geq 2$?}

\item{Suppose that $G$ is an HNN extension with stable letter $t$, 
base $H$ and associated subgroup $F\subset H$. 
Is $\mu (G)=\cap t^k Ft^{-k}$ finitely generated?}
(See \cite{[Ka89]} for a related result on knot groups, and also \cite{[So91]}.)
\end{enumerate}                            

\section{residual finiteness, hopficity, cohopficity}

Let $K_n =\cap\{ H\subset G|[G:H]~ divides ~ n!\}$. 
Then $[G:K_n ]$ is finite, for all 
$n\geq 1$, and $G$ is residually finite if and only if $\cap K_n =1$. 
If $G$ is not virtually representable onto $\mathbb{Z}$ this intersection is also 
the intersection of the terms in the more rapidly descending series given by 
$K_n^{(n)} $, and is contained in $G^{(\omega )} $.

If $G$ has a maximal finite $p$-quotient $P$ for some prime $p$
then $P$ has cohomological period dividing 4,
and so is cyclic, if $p$ is odd, and cyclic or quaternionic, if $p=2$.
Hence if $\beta_1(G;\mathbb{F}_p)>1$ for some odd prime $p$,
or if $\beta_1(G;\mathbb{F}_2)>2$,
then the pro-$p$ completion of $G$ is infinite \cite{[Me90']}.

If  $[G:\cap{K_n}]=\infty$ and $G$ is a 3-manifold group then either 
$G$ is solvable or there is a prime $p$ such that $G$ has subgroups $H$ 
of finite index with $\beta_1(H;\mathbb{F}_p)$ arbitrarily large \cite{[Me90']}.
Hence either some such $H$ maps onto $\mathbb{Z}$ or the pro-$p$ completion of any 
such subgroup with $\beta_1(H;\mathbb{F}_p)>1$ is a pro-$p$ $PD_3$-group
 \cite{[KZ08]}. 
 If $G$ is almost coherent and  $[G:\cap{K_n}]=\infty$ then it satisfies the Tits alternative
 \cite{bb17}.

The groups of 3-manifolds are residually finite, 
by \cite{[He87]} and the Geometrization Theorem. 
Hence they are hopfian, 
i.e., onto endomorphisms of such groups are automorphisms. 
The Baumslag-Solitar groups $\langle x,t\mid tx^pt^{-1}=x^q\rangle$ 
embed in $PD_4$-groups.
Since these groups are not hopfian, 
there are $PD_4$-groups which are not residually finite \cite{[Me90]}. 
No such Baumslag-Solitar relation with $|p|\not=|q|$ holds in any $PD_3$-group $G$;
there is no homomorphism from $\langle x,t\mid tx^pt^{-1}=x^q\rangle$ to $G$ \cite{ca07}. 

Let $\mathcal{X}$ be the class of groups of cohomological dimension 2 which have 
an infinite cyclic subgroup which is commensurate with all of its conjugates.
If $G$ is a $PD_3$-group with no nontrivial abelian normal subgroup and which 
contains a subgroup isomorphic to $\mathbb{Z}^2$ then $G$ splits over an 
$\mathcal{X}$-group \cite{[Kr93]}. (See also \cite{[Kr90b]}.)
This class includes the Baumslag-Solitar groups and also the fundamental groups 
of Seifert fibred 3-manifolds with nonempty boundary.
If $H$ is in the latter class and is not virtually $\mathbb{Z}^2$
then $\sqrt{H}\cong\mathbb{Z}$ and $H/\sqrt{H}$ is a free product of cyclic groups.
It then follows from the result of \cite{ca07} on Baumslag-Solitar relations
that finitely generated $\mathcal{X}$-groups which are subgroups of $PD_3$-groups 
are of this ``Seifert type".

An injective endomorphism of a $PD_3$-group must have image of finite index, 
by Strebel's theorem \cite{[St77]}.
A 3-manifold group satisfies the {\it volume condition}
(isomorphic subgroups of finite index have the same index)
if and only if it is not solvable and is not virtually a product \cite{[WW94], [WY99]}.
In particular, such 3-manifold groups are cohopfian,
i.e., injective endomorphisms are automorphisms.
The volume condition is a property
of commensurability classes; this is not so for cohopficity.

\begin{enumerate}
\addtocounter{enumi}{37}

\item{Does every $PD_3$-group have a proper subgroup of finite index?}

\item{Are all $PD_3 $-groups residually finite?}

\item{Let $\widehat{G}$ be a pro-$p$ $PD_3$-group. Is $G$ virtually
representable onto $\widehat{Z}_p$?} 

\item{Do all $PD_3$-groups other than those which are solvable or are virtually products
satisfy the volume condition?}
\end{enumerate}

\section{other questions}

We conclude with some related questions. 

\begin{enumerate}  
\addtocounter{enumi}{41}

\item{Let $P$ be an indecomposable, non-orientable $PD_3$-complex.
If the orientable double cover $P^+$ is homotopy equivalent to a 3-manifold,
is $P$ itself homotopy equivalent to a 3-manifold?}
                                                  
\item{Let $X$ be a $PD_3$-complex.
Is $X\times{S^1}$ or $X\times{S^1}\times{S^1}$ homotopy equivalent to a closed
manifold?}

\item{Is there an explicit example of a free action of a generalized quaternionic
group $Q(8a,b,1)$ (with $a,b>1$ and $(a,b)=1$) on an homology 3-sphere?}

\item{Is there a purely algebraic analogue of orbifold 
hyperbolization which may be used to show that every $FP$ group of 
cohomological dimension $k$ is a subgroup of a $PD_{2k}$-group?}

\end{enumerate}

See \cite{[Hi04]} and the references there 
\cite{[GW], [GW94], [Wn91], [Wn93]}
for work on maps of nonzero degree between $PD_3$-groups.


\begin{thebibliography}{99} 

\bibitem{[Ag12]} Agol, I. The Virtual Haken Conjecture,
with an appendix by Agol, D. Groves and J. Manning,

Doc. Math. 18 (2013), 1045--1087.

\bibitem{BB08} Baues, H.-J., and Bleile, B., \ Poincar{\'{e}} duality complexes in dimension four, \  

Alg. Geom. Top. {\bf 8} (2008), 2355--2389. 

\bibitem{BS89} Baumslag, G. and Shalen, P.B. Groups whose three-generator subgroups are free,

Bull. Austral. Math. Soc. 40 (1989), 163--174.
                       
\bibitem{[Be96]} Bestvina, M. Local homological properties of boundaries of groups,

Michigan Math. J. 43 (1996), 123--139.

\bibitem{bm91} Bestvina, M. and Mess, G. The boundary of negatively curved groups,

J. Amer. Math. Soc. 4 (1991), 469--481.
                                    
\bibitem{[B]} Bieri, R. {\it Homological Dimensions of Discrete Groups}, 

Queen Mary College Mathematical Notes, London (1976).

\bibitem{be78} Bieri, R. and Eckmann, B. Relative homology and Poincar\'e duality for group pairs,

J. Pure Appl. Algebra 13 (1978), 279--313.

\bibitem{[BH91]} Bieri, R. and Hillman, J.A. Subnormal subgroups in 3-dimensional 
Poincar\'e duality groups, Math. Z. 206 (1991), 67--69.

\bibitem{[BNS87]} Bieri, R., Neumann, W.D. and Strebel, R. A geometric invariant of
discrete groups,

Inventiones Math. 90 (1987), 451--477.

\bibitem{[BS78]} Bieri, R. and Strebel, R. Almost finitely presentable soluble groups,

Comment. Math. Helv. 53 (1978), 258--278.

\bibitem{bl} Bleile, B. Poincar\'e Duality Pairs of Dimension Three,

Forum Math. 22 (2010), 277 -- 301.

\bibitem{bbh} Bleile, B., Bokor, I. and Hillman, J.A. Poincar\'e complexes 
with highly 

connected universal covers,
arXiv: 1605.00096 [math.GT] (23 pages).

\bibitem{bb17} Boileau, M. and Boyer, S. On the Tits alternative for $PD_3$-groups,

arXiv: 1710.10670 [math.GT] (16 pages).

\bibitem{bow} Bowditch, B.H. Planar groups and the Seifert conjecture,

J. Reine Angew. Math. 576 (2004), 11--62.

\bibitem{browd} Browder, W. Poincar\'e complexes, their normal fibrations
and surgery,

Invent. Math. 17 (1972), 191--202.

\bibitem{brown} Brown, K.S. {\it Cohomology of Groups},

Graduate Texts in Mathematics 87,

Springer-Verlag, Berlin -- Heidelberg -- New York (1982).

\bibitem{[Br75]} Brown, K.S. A homological criterion for finiteness, 

Comment. Math. Helv. 50 (1975), 129--135.

\bibitem{[BG85]} Brown, K.S. and Geoghegan, R. 
Cohomology with free coefficients of
the fundamental group of a graph of groups,
Comment. Math. Helv. 60 (1985), 31--45.        

\bibitem{ca07} Castel, F. Centralisateurs d'\'el\'ements dans les $PD_3$-paires,

Comment. Math. Helv. 82 (2007), 499--517.

\bibitem{crisp} Crisp, J.S. The decomposition of Poincar\'e duality complexes,

Comment. Math. Helv. 75 (2000), 232--246.

\bibitem{cr2} Crisp, J.S. An algebraic loop theorem and the decomposition
of $PD^3$-pairs,

Bull. London Math. Soc. 39 (2007), 46--52.

\bibitem{da00} Davis, M. Poincar\'e duality groups,

in {\it Surveys on Surgery Theory\/} (edited by S. Cappell, A. Ranicki and J. Rosenberg), vol. 1,

Ann. Math. Study 145, 
Princeton University Press, 
Princeton, N.J. (2000), 167--193.

\bibitem{dd} Dicks, W. and Dunwoody, M.J. {\it Groups acting on Graphs},

Cambridge studies in advanced mathematics 17,

Cambridge University Press, Cambridge (1989).

\bibitem{[Du89]} Dunwoody, M.J. Bounding the decomposition of a Poincar\'e duality group,

Bull. London Math. Soc. 21 (1989), 466--468.

\bibitem{ds00} Dunwoody, M.J. and Swenson, E. The algebraic torus theorem,

Inventiones Math. 140 (2000), 605--637.
                                                                                      
\bibitem{[Ec86]} Eckmann, B. Cyclic homology of groups and the Bass conjecture,

Comment. Math. Helv. 61 (1986), 193--202.

\bibitem{el81} Eckmann, B. and Linnell, P.A. Poincar\'e duality groups of dimension two, II,

Comment. Math. Helv. 58 (1981), 111--114.

\bibitem{em80} Eckmann, B. and M\"uller, H. Poincar\'e duality groups of dimension two,

Comment. Math. Helv. 55 (1980), 510--520.

\bibitem{Ed} Edmonds, A.L. Deformation of maps to branched coverings in dimension two,

Ann. Math. 110 (1979), 113--125.

\bibitem{[El84]} Elkalla, H.S. Subnormal groups in 3-manifold groups, 

J. London Math. Soc. 30 (1984), 342--360.
    
\bibitem{[FH99]} Feighn, M. and Handel, M. Mapping tori of free group automorphisms are coherent,

Ann. Math. 149 (1999), 1061--1077.

\bibitem{Gad} Gadgil, S. Degree-one maps, surgery and four-manifolds,

Bull. London Math. Soc. 39 (2007), 419--424.

\bibitem{[GS]} Gersten, S.M. and Stallings, J.R. (editors)
 {\it Combinatorial Group Theory and Topology}, 

Ann. Math. Study 111, 
Princeton University Press, Princeton (1987).

\bibitem{[Gi79]} Gildenhuys, D.  Classification of soluble groups of cohomological dimension
two,

Math. Z. 166 (1979), 21--25.
   
\bibitem{[GS81]} Gildenhuys, D. and Strebel, R. On the cohomological dimension 
of soluble groups, 

Canadian Math. Bull. 24 (1981), 385--392.

\bibitem{[GW]} Gonz\'alez-Acu\~na, F. and Whitten, W. {\it Imbeddings of 
Three-manifold Groups}, 

American Mathematical Society Memoir 474 (1992).
 
\bibitem{[GW94]} Gonz\'alez-Acu\~na, F. and Whitten, W. 
Cohopficity of 3-manifold groups,

Top. Appl. 56 (1994), 87--97.

\bibitem{hm86} Hambleton, I. and Madsen, I. Actions of finite groups on $R^{n+k}$ 
with fixed point set $R^k$,
Canad. J. Math. 38 (1986), 781--860.

\bibitem{[He87]} Hempel, J. Residual finiteness for 3-manifolds, 
in \cite{[GS]}, 379--396.

\bibitem{he77} Hendriks, H. Obstructions theory in  3-dimensional topology: 
an extension theorem,

J. London Math. Soc. 16 (1977), 160--164.

\bibitem{Hi} Hillman, J.A. {\it Four-Manifolds, Geometries and Knots\/},

GT Monograph 5, Geom.  Topol. Publ., Coventry (2002; revisions 2007,  2014).

\bibitem{hi85} Hillman, J.A. Seifert fibre spaces and Poincar\'e duality groups, 

Math. Z. 190 (1985), 365--369.

\bibitem{[Hi87]} Hillman, J.A. Three-dimensional Poincar\'e duality groups which 
are extensions, 

Math. Z. 195 (1987), 89--92.

\bibitem{hi93} Hillman, J.A. On 3-dimensional Poincar\'e duality complexes
and 2-knot groups,

Math. Proc. Cambridge Philos. Soc. 114 (1993), 215--218.

\bibitem{[Hi96]} Hillman, J.A.  Embedding homology equivalent $3$-manifolds in $4$-space,

Math. Z. 223 (1996), 473--481.

\bibitem{[Hi03]} Hillman, J.A. Tits alternatives and low dimensional topology,

J. Math. Soc. Japan 55 (2003),365--383.

\bibitem{hi04} Hillman, J.A. An indecomposable $PD_3$-complex: II,

Alg. Geom. Topol. 4 (2004), 1103--1109.

\bibitem{[Hi04]} Hillman, J.A. Homomorphisms of nonzero degree between 
$PD_n$-groups,

J. Austral. Math. Soc. 77 (2004), 335--348.

\bibitem{hi06} Hillman, J.A. Centralizers and normalizers of subgroups of
$PD_3$-groups and group pairs,

J. Pure Appl. Alg. 204 (2006), 244--257.

\bibitem{hi12}  Hillman, J.A. Indecomposable $PD_3$-complexes,

Alg. Geom. Topol. 12 (2012), 131--153.

\bibitem{hi17}  Hillman, J.A. Indecomposable non-orientable $PD_3$-complexes,

Alg. Geom. Topol. 17 (2017), 645--656.

\bibitem{[HL92]} Hillman, J.A. and Linnell, P.A. Elementary amenable groups 
of finite Hirsch length are locally finite by virtually solvable, 
J. Austral. Math. Soc. 52 (1992), 237--241.

\bibitem{[Ho82]} Howie, J. On locally indicable groups, 

Math. Z. 180 (1982), 445--461.

\bibitem{jw72} Johnson, F.E.A. and Wall, C.T.C. On groups satisfying Poincar\'e duality,

Ann. Math. 69 (1972), 592--598.
 
\bibitem{[KM10]} Kahn, J. and Markovic, V. Immersing almost geodesic surfaces 
in a closed hyperbolic three manifold, Ann. Math. 175 (2012), 1127--1190.
   
\bibitem{[Ka89]} Kakimizu, O. On maximal fibred submanifolds of a knot exterior, 

Math. Ann. 284 (1989), 515--528.

\bibitem{kk05} Kapovich, M. and Kleiner, B. Coarse Alexander duality 
and duality groups,

J. Diff. Geom. 69 (2005), 279--352.

\bibitem{kk07} Kapovich, M. and Kleiner, B. The weak hyperbolization 
conjecture for

3-dimensional CAT(0)-groups,
Groups Geom. Dyn. 1 (2007), 61--79.

\bibitem{[KK07]} Kapovich, M. and Kleiner, B. The weak hyperbolization conjecture
for 3-dimensional $CAT(0)$-groups, Groups Geom. Dyn. 1 (2007), 61--79.

\bibitem{kn13} Kar, A. and  Niblo, G.A.   A topological splitting theorem for Poincar\'e duality 
groups and high dimensional manifolds,
Geom. Topol. 17 (2013), 2203--2221.

\bibitem{kz07} Kochloukova, D.H. and Zaleskii, P.A. Tits alternative for 3-manifold groups,

Arch. Math. (Basel) 88 (2007), 364--367.

\bibitem{[KZ08]} Kochloukova, D.H. and Zaleskii, P. A.
Profinite and pro-$p$ completions of Poincar\'e duality groups of dimension 3,
Trans. Amer. Math. Soc. 360 (2008), 1927--1949.

\bibitem{[Ko87]} Kojima, S. Finite covers of 3-manifolds containing essential 
surfaces of Euler characteristic 0, Proc. Amer. Math. Soc. 101 (1987), 743--747.

\bibitem{Kr90} Kropholler, P. H.  An analogue of the torus decomposition 
theorem for certain Poincar\'e duality groups,
Proc. London Math. Soc. 60 (1990), 503--529.

 \bibitem{[Kr90b]} Kropholler, P. H. Baumslag-Solitar groups and some other 
groups of cohomological dimension two, Comment. Math. Helv. 65 (1990), 547--558.
    
\bibitem{[Kr93]} Kropholler, P. H. A group theoretic proof of the torus theorem, 

in  \cite{[NR]}, 138--158.

\bibitem{[KR88]} Kropholler, P. H. and Roller, M. A. Splittings of Poincar\'e 
duality groups, 

Math. Z. 197 (1988), 421-438.                               

\bibitem{[KR88a]} Kropholler, P. H. and Roller, M. A. Splittings of Poincar\'e 
duality groups II, 

J. London Math. Soc. 38 (1988), 410--420.

\bibitem{kr89} Kropholler, P. H. and Roller, M. A. Splittings of Poincar\'e 
duality groups, III,

J. London Math. Soc. 39 (1989), 271--284.

\bibitem{[KR89a]} Kropholler, P. H. and Roller, M. A. Relative ends and duality groups.

J. Pure Appl. Alg. 61 (1989), 197--210.

\bibitem{mar12} Markovic, V. Criteria for Cannon's conjecture,

Geom. Funct. Anal. 23 (2012), 1035-1061.

\bibitem{[Me90]} Mess, G. Examples of Poincar\'e duality groups, 

Proc. Amer. Math. Soc. 110 (1990), 1144--5.

\bibitem{[Me90']} Mess, G. Finite covers of 3-manifolds and a theorem of Lubotzky,

preprint, I.H.E.S. (1990).

\bibitem{mil} Milnor, J.W. Groups which act on $S^n$ without fixed points,

Amer. J. Math. 79 (1957), 623-630.

\bibitem{[NR]} Niblo, G.A. and Roller, M.A. (editors) {\it Geometric Group Theory 1,} 

London Mathematical Society Lecture Note Series 181, 

Cambridge University Press, Cambridge (1993).

\bibitem{[PW12]} Przytycki, P. and Wise, D.T. Mixed 3-manifolds are virtually special,

arXiv: 1205.6742 [math.GR].

\bibitem{sa11} Sapir, M. Aspherical groups and manifolds with extreme properties,

arXiv: 1103.3873v3 [math.GR].

\bibitem{[Sc73]} Scott, G.P. Compact submanifolds of 3-manifolds, 

J. London Math. Soc. 1 (1973), 437--440.

\bibitem{[Sc76]} Scott, G.P. Normal subgroups in 3-manifold groups, 

J. London Math. Soc. 13 (1976), 5--12.

\bibitem{ss03} Scott, P. and Swarup, G. Regular neighbourhoods and canonical 

decompositions for groups,
Ast\'erisque 289 (2003).

\bibitem {[So91]} Soma, T. Virtual fibre groups in 3-manifold groups, 

J. London Math. Soc. 43 (1991), 337--354.

\bibitem{[St77]} Strebel, R. A remark on subgroups of infinite index in 
Poincar\'e duality groups,

Comment. Math. Helv. 52 (1977), 317--324.

\bibitem{[S]} Stallings, J.R. {\it Group Theory and 3-Dimensional Manifolds}, 

Yale Mathematical Monograph 4, 

Yale University Press, New Haven - London (1971).

\bibitem{swan} Swan, R.G. Periodic resolutions for finite groups,

Ann. Math. 72 (1960), 267--291.

\bibitem{swar74} Swarup, G. On a theorem of C.B.Thomas,

J. London Math. Soc. 8 (1974), 13--21.

\bibitem{[Te97]} Teichner, P. Maximal nilpotent quotients of 3-manifold groups,

Math. Res. Lett. 4 (1997), 283--293.
                                                                            
\bibitem{thom67} Thomas, C.B. The oriented homotopy type of compact 3-manifolds,

Proc. London Math. Soc. 19 (1967), 31--44.

\bibitem{thom84} Thomas, C.B. Splitting theorems for certain $PD^3$-groups,

Math. Z. 186 (1984), 201--209.

\bibitem{tw17} Tshishiku, B. and Walsh, G. On groups with $S^2$ Bowditch boundary,

arXiv: 1710.09018 [math.GT].

\bibitem{tur} Turaev, V.G. Three dimensional Poincar\'e duality complexes: splitting and classification,
Mat. Sbornik 180 (1989), 809--830.

English translation: Math. USSR-Sbornik 67 (1990), 261--282.
    
\bibitem{wall} Wall, C.T.C. Poincar\'e complexes: I,

Ann. Math. 86 (1967), 213--245.

\bibitem{wall04} Wall, C.T.C. Poincar\'e duality in dimension 3,

in {\it Proceedings of the Casson Fest},

GT Monogr. 7, Geom. Topol. Publ., Coventry (2004), 1--26.

\bibitem{[Wn91]} Wang, S. The existence of maps of nonzero degree between 
aspherical 3-manifolds, 

Math. Z. 208 (1991), 147--160.
    
\bibitem{[Wn93]} Wang, S. The $\pi_1 $-injectivity of self-maps of nonzero degree 
on 3-manifolds, 

Math. Ann. 297 (1993), 171--189.
                       
\bibitem{[WW94]} Wang, S. and Wu, Y.Q. Covering invariants and co-hopficity of
3-manifold groups,

Proc. London Math. Soc. 68 (1994), 203-224.

\bibitem{[WY99]} Wang, S. and Yu, F. Covering degrees are determined by graph
manifolds involved,

Comment. Math. Helv. 74 (1999), 238--247.

\bibitem{Z} Zieschang, H. {\it Finite Groups of Mapping Classes of surfaces},

Lecture Notes in Mathematics 875,

Springer-Verlag, Berlin -- Heidelberg -- New York (1981).
    
\bibitem{zi82} Zimmermann, B. Das Nielsensche Realisierungsproblem 
f\"ur hinreichend grosse 

3-Mannigfaltigkeiten,
Math. Z. 180 (1982), 349--359.

\end{thebibliography}
\end{document}